\documentclass[a4paper,12pt]{amsart}

\usepackage{enumerate,layout,amssymb,epsf,mathrsfs}%,pxfonts}
\usepackage{eepic,graphicx,amscd}
\usepackage{wrapfig}

\newtheorem{theorem}{Theorem}[section]

\newtheorem{proposition}[theorem]{Proposition}

\newtheorem{lemma}[theorem]{Lemma}

\theoremstyle{definition}

\newtheorem{definition}[theorem]{Definition}

\newcommand{\Z}{{\mathbb Z}}
\newcommand{\N}{{\mathbb N}}

\newcommand{\R}{{\mathbb R}}
\newcommand{\B}{{\mathcal B}}
\newcommand{\C}{{\mathcal C}}

\renewcommand{\L}{{\mathcal L}}

\DeclareMathOperator{\Orb}{Orb}

\DeclareMathOperator{\id}{id}

\numberwithin{equation}{section}

\begin{document}

\title[uniform sets for transformations]{Uniform sets for 
infinite measure-preserving systems}

\author[H. Yuasa]{Hisatoshi Yuasa}
\address{17-23-203 Idanakano-cho, Nakahara-ku, Kawasaki Kanagawa 211-0034, JAPAN.}
%\dedicatory{Dedicated to Kiyoshiro Imawano}

\email{hisatoshi.yuasa@gmail.com}

\begin{abstract}
The concept of a uniform set is introduced for an ergodic, measure-preserving 
transformation on  a 
non-atomic, infinite Lebesgue space. The uniform sets exist as much 
as they generate the underlying $\sigma$\nobreakdash-algebra. This leads to the 
result that any ergodic, measure-preserving transformation on a non-atomic, 
infinite Lebesgue space is isomorphic to a minimal homeomorphism on a locally 
compact metric space which admits a unique, up to scaling, invariant Radon measure.
\end{abstract}



\maketitle

\section{Introduction}\label{intro}

The present work concerns constructing a topological model of a given ergodic 
measure-preserving system. R.~I.~Jewett \cite{Jewett} showed that 
any weakly mixing 
measure\nobreakdash-preserving transformation on a non\nobreakdash-atomic, 
Lebesgue probability space is (measure\nobreakdash-theoretically) isomorphic 
to a strictly ergodic homeomorphism on a Cantor set, which was extended by 
W.~Krieger~\cite{Krieger} to all ergodic systems. This 
model theorem is the so\nobreakdash-called Jewett-Krieger Theorem. 
G.~Hansel and J.~P.~Raoult \cite{HR} emphasized that uniform partitions 
play important roles in proving Jewett-Krieger Theorem. 
B.~Weiss~\cite{W1,W2} added a categorical taste to the model theorem: if 
$\pi:{\mathbf Y}_1 \to {\mathbf Y}_2$ is a factor map between ergodic systems 
and if ${\mathbf X}_2$ is a strictly ergodic 
model of ${\mathbf Y}_2$, then there exists a strictly ergodic model 
${\mathbf X}_1$ of ${\mathbf Y}_1$ such that the diagram
\[
\begin{CD}
{\mathbf Y}_1 @> \phi >> {\mathbf X}_1\\
@V \pi VV @VV \rho V \\
{\mathbf Y}_2 @>> \psi > {\mathbf X}_2
\end{CD}
\]
commutes, where $\phi$ and $\psi$ are isomorphisms and $\rho$ is a 
topological factor map. In connection with topological orbit equivalence 
\cite{GPS}, N.~Ormes~\cite{Ormes} showed  a generalization of 
Jewett-Krieger Theorem that any ergodic system has a topological model which is 
orbit equivalent to a given Cantor minimal system with a given invariant 
probability measure. Along this line, H.~Matui~\cite{M} achieved a model 
theorem which realizes an ergodic system as a minimal homeomorphisms on a locally 
compact metric space. I.~Kornfeld  and N.~Ormes \cite{KO} generalized 
Ormes' model theorem for families of ergodic systems. 
Strictly ergodic models for ergodic actions by groups other than $\Z$ were 
obtained for $\R$\nobreakdash-action by K.~Jacobs~\cite{Ja}, 
M.~Denker and E.~Eberlein~\cite{DE}, and for free actions by the 
commutative groups by B.~Weiss~\cite{W1}. 

All of the above\nobreakdash-mentioned model theorems target the ergodic 
transformations preserving probability measures. It is the present work that 
initiates a model theorem of ergodic, infinite 
measure\nobreakdash-preserving systems. It is actually proved in 
Theorem~\ref{J-K} that an ergodic, infinite measure\nobreakdash-preserving 
system has a topological model of a minimal homeomorphism on a locally compact 
metric space admitting a unique, up to scaling, invariant Radon measure. 
Since our strategy for proving the theorem is to follow the line of 
B.~Weiss~\cite{W1,W2} (see also \cite{G}), the concept of a uniform set should 
be formalised suitably for the infinite measure\nobreakdash-preserving case. 
This is accomplished by Definition~\ref{uniform_infi}. 
Lemmas~\ref{initial_partition} and \ref{refining_uniform} are crucial to prove 
Theorem~\ref{J-K}. It follows from the lemmas that the ergodic system can be 
approximated by a refining sequence of uniform partitions which generate the 
underlying $\sigma$\nobreakdash-algebra. A uniform partition is a finite 
partition with a unique atom of infinite measure whose atoms of finite measure 
are all uniform. The uniform 
partition has an advantage to give rise to an almost minimal factor admitting 
a unique, up to scaling, invariant Radon measure. This fact is 
verified in virtue of Proposition~\ref{existence}. The proposition 
characterizes a homeomorphism on a locally compact metric space admitting a 
unique, up to scaling, Radon measure. As an immediate consequence of 
Theorem~\ref{J-K}, 
it holds that an ergodic, infinite measure\nobreakdash-preserving system is 
isomorphic to the Vershik map arising from an almost simple, ordered Bratteli 
diagram in the sense of \cite{D}. Unfortunately, a categorical realization of 
a factor map between infinite measure\nobreakdash-preserving systems has not 
been achieved yet. This problem is the infinite measure counterpart of the 
categorical model theorem of B.~Weiss~\cite{W1,W2}.

The author thanks Professor T.~Hamachi for letting him know the existence of 
an extension \cite{Okabe} of Kolmogorov consistency theorem to infinite 
measure spaces.

\section{Ratio ergodic theorem on towers}

In this section, we briefly review some basic concepts and facts concerning 
transformations on measure spaces which preserve measures; in particular, 
finite partitions, symbolic factors associated with them and 
Kakutani\nobreakdash-Rohlin partitions. Without explicitly stating, we 
assume any relations among measurable subsets of a measure space, or any 
properties of maps between measure spaces and so on are taken to hold up to 
sets of measure zero.

If a measure space $(Y,\C,\nu)$ is isomorphic 
to the measure space of real numbers equipped with the 
$\sigma$\nobreakdash-algebra of Lebesgue measurable subsets and Lebesgue 
measure, then we call $(Y,\C,\nu)$ a {\em non\nobreakdash-atomic, infinite 
Lebesgue space}. If $X$ is a complete separable metric 
space and $\mu$ is a non\nobreakdash-atomic, infinite, 
$\sigma$\nobreakdash-finite measure on the Borel 
$\sigma$\nobreakdash-algebra $\B$ of $X$, then $(X,\B_\mu,\mu)$ is 
isomorphic to a non\nobreakdash-atomic, infinite Lebesgue space, where 
$\B_\mu$ is the completion of $\B$ under $\mu$. See for details 
\cite[Chapter~1]{Aa}.

Let $(Y,\C,\nu)$ be a non\nobreakdash-atomic, infinite Lebesgue space. A 
bijection $T:Y \to Y$ is said to be {\em bi\nobreakdash-measurable} if both of 
$T$ and $T^{-1}$ are measurable. Suppose a bi\nobreakdash-measurable 
bijection $T:Y \to Y$ {\em preserves the measure} $\nu$, i.e.\ 
$\nu \circ T^{-1}(E):=\nu(T^{-1}E) = \nu(E)$ for all $E \in \C$. The measure 
$\nu$ is also said to be {\em $T$\nobreakdash-invariant}. We then call 
$(Y,\C,\nu,T)$ an {\em infinite measure\nobreakdash-preserving system}. 
If $\nu(B)=0$ or $\nu(Y \setminus B) = 0$ for any $T$\nobreakdash-invariant set 
$B \in \C$, then $T$ is said to be {\em ergodic}. 
The ergodicity implies that $T$ is {\em aperiodic}, i.e.\ 
the {\em orbit} $\Orb_T(y):=\{T^ny|n \in \Z\}$ of any point $y \in Y$ is 
infinite. We refer to a set $\{T^iy|m \le i \le n\}$ with $m \le n$ as a 
{\em section} of $\Orb_T(y)$. 
Symbolic examples over finite states of ergodic, infinite 
measure\nobreakdash-preserving systems are recently obtained by 
\cite{BKM,BKMS,HY,Y4}. 
Throughout the remainder of this paper, we assume $(Y,\C,\nu,T)$ is an 
ergodic, infinite measure\nobreakdash-preserving system. 
Suppose $(Y^\prime,{\mathcal C}^\prime,\nu^\prime,T^\prime)$ is an infinite 
measure\nobreakdash-preserving system. If there exists a measurable surjection 
$\phi:Y \to Y^\prime$ such that $\nu^\prime = \nu \circ \phi^{-1}$ and 
$\phi \circ T = T^\prime \circ \phi$, then $\phi$ and $T^\prime$ are called a 
{\em factor map} and a {\em factor} of $T$, respectively. 
Then, $T^\prime$ is necessarily ergodic. If in addition $\phi$ is injective, 
then $(Y^\prime,\C^\prime,\nu^\prime,T^\prime)$ is said to be {\em isomorphic} to 
$(Y,\C,\nu,T)$.

An element of a partition of $Y$ into measurable subsets is called an 
{\em atom}. A partition $\beta$ is said to {\em refine} a partition 
$\alpha$ if each atom of $\alpha$ is a union of atoms of $\beta$. If $\alpha$ 
is written as $\{A_1,A_2,\dots,A_N\}$ with $N \ge 2$ and if $\alpha$ has a 
unique atom of infinite measure, then we call $\alpha$ a {\em finite partition}. 
We always assume the unique atom has index $1$, i.e.\ $\nu(A_1)=\infty$. 
Let $K_{\alpha}$ denote the set $Y \setminus A_1$. 
For another finite partition $\beta$, we define the {\em join} 
$\alpha \vee \beta$ to be a finite partition 
$\{A \cap B|A \in \alpha, B \in \beta\}$. This definition and notation may be 
extended naturally to the join of a finite number of finite partitions. A finite 
partition $\bigvee_{i=m}^nT^{-i}\alpha$ with $m \le n$ is denoted by 
$\alpha_m^n$. If $\beta$ is written as $\{B_1,B_2,\dots,B_N\}$, then we set 
$d(\alpha,\beta)=\sum_{i \ne 1}\nu(A_i \triangle B_i)$. 
If a sequence $\{\alpha_n\}_{n \in \N}$ of finite partitions is a Cauchy sequence 
in $d$, where $\sharp \alpha_n$ is assumed constant, then there exists a 
finite partition $\alpha_0$ such that 
$\lim_{n \to \infty}d(\alpha_n,\alpha_0)=0$.

Regard the index set ${\mathfrak A}_\alpha=\{1,2,\dots,N\}$ of the finite 
partition $\alpha$ as a finite alphabet. Define a measurable map 
$\phi_\alpha:Y \to {\mathfrak A}_\alpha^\Z$ so that 
$T^iy \in A_{\phi_\alpha(y)_i}$ for every $i \in \Z$. 
An infinite, $\sigma$\nobreakdash-finite, Borel measure 
$\hat{\mu}_\alpha:=\nu \circ {\phi_\alpha}^{-1}$ is invariant under the left 
shift on ${\mathfrak A}_\alpha^\Z$. The support $\hat{X}_\alpha$ of 
$\hat{\mu}_\alpha$, i.e.\ the smallest closed subset of full measure, is 
a shift\nobreakdash-invariant Cantor set. 
The map $\phi_\alpha$ works as a factor map from $(Y,\nu,T)$ to an ergodic, 
infinite measure\nobreakdash-preserving system $(\hat{X}_\alpha,\hat{\mu}_\alpha,
\hat{S}_\alpha)$, where $\hat{S}_\alpha$ is the restriction of the left shift 
to $\hat{X}_\alpha$. Set 
\[
{\mathcal L}(\alpha) = \bigcup_{n=0}^\infty \{ w:=w_1w_2 \dots w_n \in 
{\mathfrak A}_\alpha^n| \nu(\bigcap_{i=1}^nT^{-(i-1)}A_{w_i}) > 0 \}.
\]
It follows from the definition 
of $\hat{X}_\alpha$ that for any $x=(x_i)_i \in {\mathfrak A}_\alpha^\Z$, 
$x \in \hat{X}_\alpha$ if and only if $x_{[-n,n)}:=x_{-n}x_{-n+1} \dots x_{n-1} 
\in {\mathcal L}(\alpha)$ for all $n \in \N$. Hence, an element $1^\infty$ of 
${\mathfrak A}_\alpha^\Z$ all of whose coordinates are $1$ belongs to 
$\hat{X}_\alpha$. Since $\phi_\alpha^{-1}(1^\infty)= \bigcap_{i \in \Z}T^{-i}A_1$ 
is $T$\nobreakdash-invariant, we obtain $\hat{\mu}_\alpha(\{1^\infty\})=0$. 
Consequently, the map $\phi_\alpha$ also works as a factor map from $(Y,\nu,T)$ 
to the restriction $(X_\alpha,\mu_\alpha,S_\alpha)$ of 
$(\hat{X}_\alpha,\hat{\mu}_\alpha,\hat{S}_\alpha)$ to a locally compact subshift 
$X_\alpha:=\hat{X}_\alpha \setminus \{1^\infty\}$. With words $u,v$, we 
associate a {\em cylinder set}: 
\[
[u.v]=\{x=(x_i)_i \in \hat{X}_\alpha| x_{[-|u|,|v|)}=uv\}, 
\]
where $|u|$ is the length of $u$. 
If $u$ is the empty word, then $[u.v]$ is abbreviated to $[v]$. The family of 
cylinder sets generate the topology of $\hat{X}_\alpha$. Given words $u$ and $v$ 
over ${\mathfrak A}_\alpha$, $\hat{\mu}_\alpha([u.v])>0$ if and only if $uv \in 
{\mathcal L}(\sigma)$, and hence, any nonempty open subset of $\hat{X}_\alpha$ 
has a strictly positive measure. Moreover, given words $u$ and $v$ satisfying 
$uv \in \L(\sigma)$, $\hat{\mu}_\alpha([u.v]) < \infty$ if and only if 
$(uv)_i \ne 1$ for some integer $i$ with $1 \le i \le |uv|$.

Let $\alpha$ and $\beta$ be finite partitions of $Y$. Suppose 
$\alpha$ is finer than $\beta$. Define a factor map 
$\phi_{\beta,\alpha}$ from $(\hat{X}_\alpha,\hat{S}_\alpha)$ to 
$(\hat{X}_\beta,\hat{S}_\beta)$ so that an atom of $\beta$ having index 
$\phi_{\beta,\alpha}(x)_i$ includes an atom of $\alpha$ having index $x_i$ for 
any $x \in \hat{X}_\alpha$ and any $i \in \Z$. Since 
$\phi_{\beta,\alpha} \circ \phi_\alpha = \phi_\beta$, we have 
$\hat{\mu}_\beta = \hat{\mu}_\alpha \circ \phi_{\beta,\alpha}^{-1}$, 
so that $(\hat{X}_\beta,\hat{\mu}_\beta,\hat{S}_\beta)$ is a factor of 
$(\hat{X}_\alpha,\hat{\mu}_\alpha,\hat{S}_\alpha)$. 
If $\gamma$ is another finite partition than which $\beta$ is finer, then it 
holds that 
$\phi_{\gamma,\alpha} = \phi_{\gamma,\beta} \circ \phi_{\beta,\alpha}$.

If $B,TB,\dots,T^{N-1}B$ are disjoint, then the family 
${\mathfrak c}:=\{B,TB,\dots,T^{N-1}B\}$ is called a {\em column} of 
{\em height} $N$ with {\em base} $B$. Each set 
$T^jB$ ($0 \le j < N$) is called a {\em level} of ${\mathfrak c}$. 
A {\em subcolumn} of ${\mathfrak c}$ is a column of the form 
$\{C,TC,\dots,T^{N-1}C\}$ with a measurable subset $C$ of $B$. A countable 
partition: 
\begin{equation}\label{FP}
{\mathfrak t}:=\{T^jB_i|0 \le j < N_i, i \in \N\}
\end{equation}
of $Y$ is called a {\em tower} with {\em base} 
$B({\mathfrak t}):=\bigcup_iB_i$. A {\em fiber} of ${\mathfrak t}$ is 
a set $\{T^jy|0 \le j < N_i\}$ with $y \in B_i$ and $i \in \N$, so that every 
fiber is a section of an orbit. The tower 
${\mathfrak t}$ is said to {\em refine} a tower ${\mathfrak t}^\prime$ if 
$B({\mathfrak t}) \subset B({\mathfrak t}^\prime)$ and if 
${\mathfrak t}$ is finer than ${\mathfrak t}^\prime$ as partitions. 

A standard way to construct a tower exploits an {\em induced transformation}. 
Given a set $B \in \C$ of positive measure, the {\em return time function} 
$r_B:B \to \N, y \mapsto \min\{n \in \N| T^ny \in B\}$ 
is well\nobreakdash-defined for a.e.\ $y \in B$, because $T$ is recurrent, or 
conservative; see \cite[Proposition~1.2.1]{Aa}. The induced 
transformation $T_B:B \to B,y \mapsto T^{r_B(y)}y$ is an ergodic, 
bi\nobreakdash-measurable bijection preserving the measure 
$\C \cap B \to {\mathbb R}_+ \cup \{\infty\},C \mapsto \nu(C)$, where 
$\C \cap B = \{B \cap A|A \in \C\}$. 
Then, $\{T^jB_i|i \in \N, 0 \le j < i\}$ is a tower, where 
$B_i=r_B^{-1}(i)$. See for details \cite[Section~1.5]{Aa}. 

A {\em Kakutani\nobreakdash-Rohlin tower}, or a {\em K\nobreakdash-R tower} for 
short, is a tower having finitely many columns. 
A K\nobreakdash-R tower is said to be {\em standard} if it has a unique 
atom of infinite measure, which we call the {\em infinite level} of the 
tower. This definition forces the infinite level to constitute a column of 
height one. If the base of a column of a standard tower has a finite measure, 
then we refer to the column as a {\em principal column}. 
We henceforth use the notation $\C_0:=\{C \in \C| 0 < \nu(C) < \infty\}$. 
\begin{definition}
Given a set $K \in \C_0$, a standard tower ${\mathfrak t}$ is said to be 
{\em $K$\nobreakdash-standard} if $K \subset K_{{\mathfrak t}}$, 
${\mathfrak t}$ refines $\{K, Y \setminus K\}$ as partitions, and 
each principal 
column of ${\mathfrak t}$ has a level included in $K$. 
\end{definition}

Let $\alpha=\{A_1,A_2,\dots,A_N\}$ be a finite partition of $Y$. {\em The 
$\alpha$\nobreakdash-name of a section} $\{T^iy|m \le i \le n\}$ is the word 
$\phi_\alpha(y)_{[m,n]}$. Let ${\mathfrak t}$ be a tower as \eqref{FP}. 
For $i \in \N$ and $w \in {\mathfrak A}_\alpha^{N_i}$, we set 
$B_{i, w}=\{y \in B_i|\phi_\alpha(y)_{[0,N_i)}=w\}$. The resulting tower:
\[
\bigcup_{i \in \N}\bigcup_{w \in {\mathfrak A}_\alpha^{N_i}}\{B_{i,w},TB_{i,w},
\dots, T^{N_i-1}B_{i,w}\} 
\]
is called the {\em refinement of ${\mathfrak t}$ according to $\alpha$}. 

The following lemmas will be significant ingredients in 
Section~\ref{constructions} for constructing {\em uniform} partitions. 

\begin{lemma}\label{tower_lemma1}
Let $N \in \N$ and $K \in \C_0$. 
Then, there exists a $K$\nobreakdash-standard tower 
${\mathfrak t}$ such that the height of every principal column of 
${\mathfrak t}$ is $N$ or $N+1$.
\end{lemma}

\begin{proof}
Take $n \in \N$ be so that every integer $ \ge n$ is written as 
\begin{equation}\label{comb}
aN + b (N+1)
\end{equation}
with $a,b \in \Z_+$. Since $T$ is ergodic, there exists a set 
$C \in \C_0$ 
such that $C,TC,\dots,T^{n-1}C$ are disjoint. Take a tower with base $C$. 
Refine the tower according to $K$. The heights of columns of the resulting 
tower are at least $n$. In view of the fact that the height of each column is 
written in the form of \eqref{comb}, divide each column into $a$ blocks of 
$N$ levels and $b$ blocks of $N+1$ levels. Of course, $a$ and $b$ depend on a 
column. Consider the tower whose columns are 
precisely these blocks. Since the heights of its columns are 
$N$ or $N+1$, by uniting all the columns whose fibers have a common 
$\{K,Y\setminus K\}$\nobreakdash-name, we change the tower into a 
K\nobreakdash-R tower. 
Unite all the levels of those columns of the K\nobreakdash-R tower, all of 
whose levels are disjoint from $K$, into a new level. 
The new level constitutes a new  column and its complement has infinite measure. 
The resulting tower is a standard K\nobreakdash-R tower with the desired 
properties. 
\end{proof}

\begin{lemma}\label{desired_refine}
Let $K \in \C_0$ and $n \in \N$. 
Suppose ${\mathfrak t}_1$ is a $K$\nobreakdash-standard tower. Then, there 
exists a $K$\nobreakdash-standard tower ${\mathfrak t}_2$ refining 
${\mathfrak t}_1$ such that the height of every principal column of 
${\mathfrak t}_2$ is between $n$ and $n+4N$, where $N$ is the largest column 
height of ${\mathfrak t}_1$.
\end{lemma}

\begin{proof}
Take $n_0 \in \N$ such that every integer $\ge n_0$ is written as 
$a(n+2N) + b(n+2N+1)$ with $a,b \in \Z_+$. Since $T$ is aperiodic, there exists 
a subset $C \in \C_0$ of $B({\mathfrak t}_1)$ such that 
$C,TC,\dots,T^{n_0-1}C$ are disjoint. Take a tower with base $C$. The heights 
of its columns are at least $n_0$. Let ${\mathfrak t}$ denote the 
refinement of the tower according to $\{K,Y \setminus K\}$. Divide each column 
of ${\mathfrak t}$ into some blocks 
of $n+2N$ levels and some blocks of $n+2N+1$ levels. The bottom level of each 
block is not necessarily included in $B({\mathfrak t}_1)$. So, we move the 
bottom level of each block to the nearest level in $B({\mathfrak t}_1)$. The 
heights of the resulting blocks are at least $n$ and at most $n+4N$. By uniting 
blocks of the same height into a column, we obtain a K-R tower 
${\mathfrak t}^\prime$ with $B({\mathfrak t}^\prime) \subset B({\mathfrak t}_1)$. 
Refine 
${\mathfrak t}^\prime$ according to ${\mathfrak t}_1$. If a column of 
${\mathfrak t}^\prime$ does not have a level included in $K$, then all the levels 
of the column are included in the infinite level of ${\mathfrak t}_1$. Unite all 
such columns of ${\mathfrak t}^\prime$ into a new level, which constitutes a new 
column. The resulting tower is the desired $K$-standard tower ${\mathfrak t}_2$. 
\end{proof}

\begin{lemma}\label{enough_hit}
Let $K \in \C_0$, $n \in \N$ and $\epsilon > 0$. Then, 
there exists $N \in \N$ for which the following holds: if 
${\mathfrak t}$ is a $K$\nobreakdash-standard tower such that 
the height of every principal column is at least $N$, then those fibers 
$\{y,Ty,\dots,T^{N_y-1}y\}$ of ${\mathfrak t}$ satisfying 
\begin{equation}\label{often}
\sum_{i=0}^{N_y-1}{\mathbf 1}_K(T^iy) \ge n
\end{equation}
covers at least $\nu(K) - \epsilon$ of $K$, where ${\mathbf 1}_K$ is the 
characteristic function of $K$. 
\end{lemma}

\begin{proof}
Since both of $T$ and $T^{-1}$ are recurrent, there exists $N_0 \in \N$ such that 
\[
F:=\{y \in K|\sum_{i=0}^{N^\prime-1}{\mathbf 1}_K(T^iy) \ge n \textrm{ and } 
\sum_{i=0}^{N^\prime-1}{\mathbf 1}_K(T^{-i}y) \ge n \textrm{ if } 
N^\prime \ge N_0\}
\]
fills up $K$, up to a set of measure $\epsilon$, i.e.\ $\nu(K \setminus F) 
< \epsilon$. Suppose a standard K\nobreakdash-R tower ${\mathfrak t}$ 
satisfies the conditions in the statement with $N=2N_0$. 
Let $A$ denote the set of fibers $\{y,Ty,\dots,T^{N_y-1}y\}$ of 
${\mathfrak t}$ satisfying \eqref{often}. Since a fiber of ${\mathfrak t}$ having 
a nonempty intersection with $F$ is included in $A$, it follows that 
$\nu(F) \le \nu(A \cap K)$, so that 
$\nu(K \setminus A) \le \nu (K \setminus F) < \epsilon$. 
This completes the proof.
\end{proof}

\begin{proposition}\label{tower_lemma2}
Let $C,K \in \C_0$ be such that $C \subset K$. 
Let $0 < \epsilon < 1$ and $M \in \N$. Then, there exists $N \in \N$ for which 
the following holds: if the height 
of each principal column of a $K$\nobreakdash-standard tower 
${\mathfrak t}$ is greater than $N$, then those fibers 
$\{y,Ty, \dots, T^{N_y-1}y\}$ of ${\mathfrak t}$, which satisfy: 
\begin{equation*}\label{approx}
\left| \frac{\sum_{i=0}^{N_y-1} {\mathbf 1}_C(T^iy)}
{\sum_{i=0}^{N_y-1} {\mathbf 1}_K(T^iy)} - 
\frac{\nu(C)}{\nu(K)}\right| < \epsilon  \textrm{ and } \sum_{i=0}^{N_y-1}
{\mathbf 1}_K(T^iy) \ge M
\end{equation*}
cover at least $\nu(K)-\epsilon$ of $K$. 
\end{proposition}

\begin{proof}
Applying Hopf's pointwise ergodic theorem (cf.\ \cite[Section~2.2]{Aa}), we may 
find $N_1 \in \N$ such that the set $F$ of those points $y \in K$ which satisfy: 
\[
\left| \frac{\sum_{i=0}^{N-1}{\mathbf 1}_C(T^iy)}
{\sum_{i=0}^{N-1}{\mathbf 1}_K(T^iy)} - \frac{\nu(C)}{\nu(K)}\right|< 
\frac{\epsilon}{3} \textrm{ if } N \ge N_1
\]
has measure at least $\nu(K) - \frac{\epsilon^2}{18}$. Choose an integer 
$N_0 > N_1$ so that $\frac{N_1}{N_0} < \frac{\epsilon}{3}$. In view of Lemma~\ref{enough_hit}, 
there exists an integer $N \ge N_0$ for which the following holds: if 
${\mathfrak t}$ is as in the statement of this lemma, 
then the set $G$ of those fibers $\{y,Ty,\dots,T^{N_y-1}y\}$ satisfying 
$\sum_{i=0}^{N_y-1}{\mathbf 1}_K(T^iy) > N_0$ covers at least 
$\nu(K) - \frac{\epsilon^2}{18}$ of $K$. 

Let ${\mathfrak t}$ be as in the statement. Put
\begin{align*}
A &=\bigcup_{y \in B({\mathfrak t})}\{y,Ty,\dots,T^{N_y-1}y| 
\sum_{i=0}^{N_y-1}{\mathbf 1}_{F \cap G}(T^iy)>
\left(1-\frac{\epsilon}{3}\right)\sum_{i=0}^{N_y-1}{\mathbf 1}_K(T^iy)\}.
\end{align*}
Since 
\begin{align*}
\int_{K \setminus A} {\mathbf 1}_{F \cap G} d \nu & = 
\sum_i \sum_{\{j|T^jB_i \subset K\}}\int_{T^jB_i \setminus A}{\mathbf 1}_{F 
\cap G} d 
\nu = \sum_i \int_{B_i \setminus A} \sum_{\{j|T^jB_i \subset K\}}{\mathbf 1}_{F 
\cap G}
\circ T^j d\nu \\
& \le \sum_i\int_{B_i \setminus A} \left(1 - \frac{\epsilon}{3}\right)
\sum_{n=0}^{N_y-1}{\mathbf 1}_K(T^ny) d\nu(y) = 
\left(1-\frac{\epsilon}{3}\right) \nu(K \setminus A),
\end{align*}
where $B_i$ is the base of a column of ${\mathfrak t}$, we obtain 
\[
\nu(F \cap G)= \int_{K \setminus A} {\mathbf 1}_{F \cap G} d\nu 
+ \int_{K \cap A} {\mathbf 1}_{F \cap G} d\nu 
\le \nu(K) - \frac{\epsilon}{3} \nu(K \setminus A). 
\]
This together with the inequality: 
$\nu(F \cap G) = \nu(K) - \nu(K \setminus (F \cap G)) 
\ge \nu(K) - \frac{\epsilon^2}{9}$ yields $\nu(K \setminus A) \le \frac{\epsilon}{3} < 
\epsilon$.

From each fiber $\{y,Ty,\dots,T^{N_y-1}y\}$ included in $A$, we may choose 
disjoint blocks $\{y_j,Ty_j,\dots,T^{h_j-1}y_j\}$, $y_j \in F$, $1 \le j \le l$, 
so that $\sum_{i=0}^{h_j-1}{\mathbf 1}_K(T^iy_j) = N_1$ for each $j$, and so that 
the union of blocks has at least 
$(1-\frac{\epsilon}{3})\sum_{i=0}^{N_y-1}{\mathbf 1}_K(T^iy)-N_1$ 
points in $K$. Let $I$ denote the complement in the fiber $\{y,Ty,\dots,T^{N_y-1}
y\}$ of the union of blocks. Using the assumption that $C \subset K$, we obtain
\begin{align*}
\left|\sum_{i=0}^{N_y-1}{\mathbf 1}_C(T^iy)-\frac{\nu(C)}{\nu(K)}
\sum_{i=0}^{N_y-1}{\mathbf 1}_K(T^iy)\right| & \le 
\sum_{j=1}^l \left| \sum_{i=0}^{h_j-1}{\mathbf 1}_C(T^iy_j) - 
\frac{\nu(C)}{\nu(K)}\sum_{i=0}^{h_j-1}{\mathbf 1}_K(T^iy_j)\right| \\
& \quad + \sum_{y^\prime \in I}\left| {\mathbf 1}_C(y^\prime) - 
\frac{\nu(C)}{\nu(K)}{\mathbf 1}_K(y^\prime)\right| \\
& \le \sum_{j=1}^l \frac{\epsilon}{3} \sum_{i=0}^{h_j-1}{\mathbf 1}_K(T^iy_j)+
\frac{\epsilon}{3} \sum_{i=0}^{N_y-1}{\mathbf 1}_K(T^iy) + N_1 \\
& \le \frac{2}{3}\epsilon \sum_{i=0}^{N_y-1}{\mathbf 1}_K(T^iy) + N_1 \\
& \le \frac{2}{3}\epsilon \sum_{i=0}^{N_y-1}{\mathbf 1}_K(T^iy) + 
\frac{\epsilon}{3}N_0,
\end{align*}
and hence $\displaystyle \left| 
\frac{\sum_{i=0}^{N_y-1}{\mathbf 1}_C(T^iy)}
{\sum_{i=0}^{N_y-1}{\mathbf 1}_K(T^iy)} - \frac{\nu(C)}{\nu(K)}\right| < \epsilon$.
\end{proof}

\begin{lemma}\label{often_lemma}
Suppose $\alpha$ is a finite partition of $Y$. Suppose a refining sequence 
$\{{\mathfrak t}_k\}_{k \in \N}$ of $K_\alpha$\nobreakdash-standard towers of 
$Y$ satisfies the properties:
\begin{enumerate}
\item\label{length_one}
${\mathfrak t}_1$ is finer than $\alpha$ as partitions;
\item\label{double}
$m_k:=\min \{ \sharp\{A \subset K_\alpha|A \in {\mathfrak c}\} | {\mathfrak c} 
\textrm{ is a principal column of } {\mathfrak t}_k\} \to \infty \textrm{ as 
} k \to \infty$. 
\end{enumerate}
Then, $\sharp(\Orb_{S_\alpha}(x) \cap K_{\phi_\alpha(\alpha)}) = \infty$ for all 
$x \in X_\alpha$.
\end{lemma}

\begin{proof}
Let $x \in X_\alpha$. If there exists a point $x^\prime \in \Orb_{S_\alpha}(x)$ 
such that 
$\sharp (\Orb_{S_\alpha}(x^\prime) \cap K_{\phi_\alpha(\alpha)})=\infty$, then 
$x$ has the same property. We may assume that $x_0 \ne 1$. For each 
$k \in \N$, let $h_k$ denote the largest height of principal columns of 
${\mathfrak t}_k$. 
There exists a sequence $\{y_k\}_{k \in \N} \subset Y$ such that for each 
$k \in \N$, the $\alpha$\nobreakdash-name $w(k)$ of 
$\{T^iy_k|-h_k \le i < h_k\}$ is identical with $x_{[-h_k,h_k)}$. Since 
$y_k \in K_{{\mathfrak t}_1} \subset K_{{\mathfrak t}_k}$ for all $k \in \N$, 
it holds that
\[ 
\sharp(\Orb_{S_\alpha}(x) \cap K_{\phi_\alpha(\alpha)}) \ge 
\sharp\{1 \le i \le 2h_k|w(k)_i \ne 1\} \ge m_k \to \infty \textrm{ as } k 
\to \infty.
\]
\end{proof}

\section{Unique invariant Radon measure}\label{sec_top}

We first define a class of topological dynamical systems which are dealt with 
often in the remainder of this paper. We also define the 
inverse limit of an inverse system consisting of the systems. 

A topological space is called a {\em Cantor set} if it is a totally disconnected, 
compact metric space without isolated points. If a homeomorphism $S$ acts on a 
Cantor set $X$, then a topological dynamical system $(X,S)$ is called a 
{\em Cantor system}. If the one\nobreakdash-point compactification 
$\hat{X}:=X \cup \{\infty\}$ of a locally compact (non\nobreakdash-compact) 
metric space $X$ is a Cantor set, then we call $X$ a 
{\em locally compact Cantor set}. It is easy to see that such a space $X$ has a 
countable base of compact open sets. A positive Borel measure 
$\mu$ on the space $X$, which is not identically zero, is called a 
{\em Radon measure} if $\mu(A) < \infty$ for any compact set $A \subset X$. If a 
homeomorphism $S:X \to X$ is {\em minimal}, i.e.\ 
$\overline{\Orb_S(x)}=X$ for all $x \in X$, 
then we call the pair $(X,S)$ a {\em locally compact Cantor minimal system}. Let 
$(X,S)$ be such a system. If a Radon measure $\mu$ on $X$ is 
$S$\nobreakdash-invariant, then $\mu(E) > 0$ for any nonempty open set $E \subset 
X$. An $S$\nobreakdash-invariant Radon measure exists; see for details \cite{Y4}. 
A unique extension $\hat{S} : \hat{X} \to \hat{X}$ of $S$ to a homeomorphism 
is {\em almost minimal} in the sense of A.~Danilenko~\cite{D}, i.e.\ $\hat{S}$ 
has a unique fixed point $\infty$ and the orbit of any other point is dense in 
$\hat{X}$. We also say that the system $(\hat{X},\hat{S})$ is almost minimal. 
Whenever we use the notation for an almost minimal system $(\hat{X},\hat{S})$, 
we suppose $(X,S)$ denote a locally compact Cantor minimal system whose unique 
extension is $(\hat{X},\hat{S})$. 
It is known that any almost minimal Cantor system is topologically conjugate 
to the Vershik map arising from an almost simple ordered 
Bratteli diagram, and vice\nobreakdash-versa; see for details \cite{D}.

Let $(\hat{X}_1,\hat{S}_1)$ and $(\hat{X}_2,\hat{S}_2)$ be almost minimal 
Cantor systems. If a continuous surjection $\phi:\hat{X}_1 \to \hat{X}_2$ 
satisfies $\hat{S}_2 \circ \phi = \phi \circ \hat{S}_1$, then $\phi$ is called a 
{\em factor map} from $(\hat{X}_1,\hat{S}_1)$ to $(\hat{X}_2,\hat{S}_2)$, and 
$\hat{S}_2$ is called a {\em factor} of $\hat{S}_1$. Let $\phi$ be the 
factor map. Let  $z_i$ denote a unique 
fixed point of $\hat{S}_i$. Since $\hat{S}_2(\phi(z_1))=\phi(z_1)$, we have 
$\phi(z_1)=z_2$. We may also verify $\phi(\hat{X}_1 \setminus \{z_1\})=
\hat{X}_2 \setminus \{z_2\}$. 
Let $\{(\hat{X}_n,\hat{S}_n)\}_{n \in \N}$ be a family of almost minimal Cantor 
systems. Suppose that for each pair $(m,n) \in \N \times \N$ with $m \ge n$, 
there exist a factor map $\phi_{n,m}$ from $(\hat{X}_m,\hat{S}_m)$ to 
$(\hat{X}_n,\hat{S}_n)$. If 
\begin{enumerate}
\item
$\phi_{n,n}=\id_{\hat{X}_n}$ for all $n \in \N$;
\item
$\phi_{n,m} \circ \phi_{m,l} = \phi_{n,l}$ for 
all $l,m,n \in \N$ with $l \ge m \ge n$,
\end{enumerate}
then we call $(\hat{X}_n,\hat{S}_n,\phi_{n,m})$ an 
{\em inverse system} of almost minimal Cantor systems. Set 
\[
\hat{X}=\{(x_n)_{n \in \N} \in \prod_{n \in \N} 
\hat{X}_n|\phi_{n,m}(x_m)=x_n \textrm{ if } m \ge n\}.
\]
Endow $\hat{X}$ with the relative topology induced by the product topology, so 
that $\hat{X}$ is a Cantor set. 
Define a homeomorphism $\hat{S}:\hat{X} \to \hat{X}$ by $(x_n)_n 
\mapsto (\hat{S}_n x_n)_n$. The homeomorphism $\hat{S}$ has a 
unique fixed point $z:=(z_n)_n$, where $z_n$ is a unique fixed 
point of $\hat{S}_n$. A topological dynamical system $(\hat{X},\hat{S})$ 
is called an {\em inverse limit} of the inverse system 
$(\hat{X}_n,\hat{S}_n,\phi_{n,m})$. Let $S$ denote the restriction 
of $\hat{S}$ to the complement $X$ of $z$ in $\hat{X}$. 
The projection $p_n:\hat{X} \to \hat{X}_n$, $(x_k)_k \mapsto x_n$ is surjective. 
It follows that $p_n^{-1}(z_n) = \{z\}$ for all $n \in \N$. 
Clearly, $p_n = \phi_{n,m} \circ p_m$ if $m \ge n$. Equip 
$\hat{X}$ with the $\sigma$\nobreakdash-algebra $\hat{\B}$ generated by an 
algebra $\bigcup_{n \in \N}p_n^{-1}\hat{\B}_n$, where $\hat{\B}_n$ is the 
Borel $\sigma$\nobreakdash-algebra of $\hat{X}_n$.

Suppose $\hat{\mu}_n$ is an $\hat{S}_n$\nobreakdash-invariant 
measure on $\hat{X}_n$, assigning zero to $\{z_n\}$, whose restriction $\mu_n$ to 
$X_n$ is a Radon measure. Assume that for all $m,n \in \N$ with $m \ge n$, 
\begin{equation}
\hat{\mu}_m \circ \phi_{n,m}^{-1} = \hat{\mu}_n.
\end{equation}
In view of Kolmogorov's extension theorem \cite{Okabe} for infinite measures, 
there exists a unique, $\sigma$\nobreakdash-finite measure $\hat{\mu}$ on 
$\hat{X}$ satisfying the property that for every $n \in \N$, 
\begin{equation}\label{kolmo}
\hat{\mu} \circ p_n^{-1} = \hat{\mu}_n. 
\end{equation}
See also \cite{Yam}. Complete $\hat{\B}$ with respect to $\hat{\mu}$. Since 
for all $n \in \N$,
\[
(\hat{\mu} \circ \hat{S}^{-1}) \circ p_n^{-1}=\hat{\mu} \circ p_n^{-1} \circ 
\hat{S}_n^{-1}= \hat{\mu}_n \circ \hat{S}_n^{-1} = \hat{\mu}_n, 
\]
it follows from \eqref{kolmo} that $\hat{\mu}$ is $\hat{S}$\nobreakdash-invariant. 
Since $p_n = \phi_{n,m} \circ p_m$ for all $m,n \in \N$ with 
$m \ge n$, we may see the family 
$\bigcup_{k \in \N}\{{p_k}^{-1}(E)|E \subset \hat{X}_k \textrm{ clopen}\}$ 
is a base of the topology of $\hat{X}$. Let $F={p_k}^{-1}(E)$ with a clopen set 
$E \subset \hat{X}_k$ and with $k \in \N$. 
Suppose $z \notin F$, i.e.\ $z_k \notin E$. 
If $F \ne \emptyset$, i.e.\ $E \ne \emptyset$, then 
$0 < \hat{\mu}(F) = \hat{\mu}_k(E) < \infty$, so that the restriction $\mu$ of 
$\hat{\mu}$ to $X$ is a Radon measure assigning a strictly positive value 
to any nonempty open set.

\begin{lemma}\label{unique_Radon}
If for each $n \in \N$, the locally compact Cantor minimal system $(X_n,S_n)$ 
has a unique, up to scaling, invariant Radon measure, then so does 
$(X,S)$ and, in addition, $S$ is minimal.
\end{lemma}

\begin{proof}
Let $\mu_n$ and $\hat{\mu}_n$ be as above, so that $\mu_n$ is the unique invariant 
Radon measure of $(X_n,S_n)$. 
Suppose $\rho$ is an $S$\nobreakdash-invariant Radon measure on $X$. Define 
a measure $\hat{\rho}$ on $\hat{X}$ by $\hat{\rho}(E)=\rho(E \cap X)$ for $E 
\in \hat{\B}$. Put $\hat{\rho}_n=\hat{\rho} \circ p_n^{-1}$. Since 
$\hat{\rho}_n$ is an $\hat{S}_n$\nobreakdash-invariant measure whose restriction 
to $X_n$ is a Radon measure and since $\hat{\rho}_n(\{z_n\})=0$, there exists 
$c_n > 0$ such that $\hat{\rho}_n = c_n \hat{\mu}_n$. Since 
$\hat{\rho}_n = \hat{\rho}_m \circ \phi_{n,m}^{-1} = c_m 
\hat{\mu}_m \circ \phi_{n,m}^{-1} = c_m \hat{\mu}_n$ 
for all $m,n \in \N$ with $m \ge n$, it follows that $c_n$ is a constant, say $c$. 
Since $(c^{-1}\hat{\rho}) \circ p_n^{-1} = \hat{\mu}_n$ for all $n \in \N$, 
the uniqueness of a measure satisfying \eqref{kolmo} yields $c^{-1}\hat{\rho} = 
\hat{\mu}$. This shows the first assertion. 
The last assertion follows from the second statement of 
Proposition~\ref{existence} below. 
\end{proof}

The next goal of this section is to prove criteria (Proposition~\ref{existence}) 
for a homeomorphism $S$ acting on a locally compact Cantor set $X$ to have a 
unique, up to scaling, invariant Radon measure. Let ${\mathfrak F}$ denote the 
ring of compact open subsets of $X$. 
For $N \in \N$, a function $f$ on a space $Z$ and a bijection $U:Z \to Z$, we set 
\[
U_Nf=\sum_{i=-N}^{N-1}f \circ U^i.
\]

\begin{proposition}\label{existence}
Assume a set $K \in {\mathfrak F}$ satisfies the property that for all $x \in X$, 
there exists $N \in \N$ such that $S_N {\mathbf 1}_K(x)>0$. 
Then, there exists an $S$\nobreakdash-invariant Radon measure. 
Then, the following conditions are mutually equivalent:
\begin{enumerate}
\item\label{uniform}
$\lim_{N \to \infty} S_N{\mathbf 1}_K(x)=\infty$ for any $x \in X$, 
and in addition, 
for any $A \in {\mathfrak F}$ and for any $\epsilon>0$, there exist 
$c \ge 0$ and $m \in \N$ such that for any $x \in K$, 
\[
S_N{\mathbf 1}_K(x) \ge m \Rightarrow 
\left|\frac{S_N {\mathbf 1}_A(x)}{S_N {\mathbf 1}_K(x)} - c\right|<\epsilon;
\]
\item\label{pointwise}
for any $A \in {\mathfrak F}$, there exists $c \ge 0$ such that 
for any $x \in K$, 
\[
\lim_{N \to \infty} \frac{S_N {\mathbf 1}_A(x)}{S_N {\mathbf 1}_K(x)}=c;
\]
\item\label{measure}
there exists an $S$\nobreakdash-invariant Radon measure $\mu$ such that 
for any $A \in {\mathfrak F}$ and for any $x \in K$, 
\[
\lim_{N \to \infty}\frac{S_N {\mathbf 1}_A(x)}{S_N {\mathbf 1}_K(x)}=
\frac{\mu(A)}{\mu(K)};
\]
\item\label{unique}
$S$ has a unique, up to scaling, invariant Radon measure.
\end{enumerate}

When these conditions hold, it holds that $S$ is minimal if and only if 
$\mu(A)>0$ for any $A \in {\mathfrak F}$, where $\mu$ is the unique invariant 
Radon measure. 
\end{proposition}

\begin{proof}
Let $\B$ denote the Borel $\sigma$\nobreakdash-algebra of $X$. 
To show the first statement, let us first assume the existence of a point 
$x_0 \in X$ such that the sequence $\{S_N{\mathbf 1}_K(x_0)\}_N$ is bounded. 
Let $A \in {\mathfrak F}$. Since $A \subset \bigcup_{i=1}^kS^{n_i}(K)$ with 
some $n_i \in \Z$, we have $\sharp (\Orb_S(x_0) \cap A) < \infty$. We may 
define a countably additive set function $\mu:{\mathfrak F} \to \R_+$ by 
$\mu(A)=\sharp(\Orb_S(x_0) \cap A)$ for $A \in {\mathfrak F}$. 
The set function $\mu$ is uniquely extended to a measure on $\B$, which 
is an $S$\nobreakdash-invariant Radon measure.

Let us then assume $\{S_N{\mathbf 1}_K(x)\}_N$ is unbounded for any $x \in 
X$. In this case, a proof is achieved by following \cite[Section~V]{F}. 
Fix points $\{x_i\}_{i \ge 1} \subset K$ and integers $1 \le N_1 <N_2<N_3<\dots$ 
so that $0 < S_{N_1}{\mathbf 1}_K(x_1) < S_{N_2}{\mathbf 1}_K(x_2) < 
S_{N_3}{\mathbf 1}_K(x_3) < \dots$ Let $A \in {\mathfrak F}$. Take $m \in \N$ 
so that ${\mathbf 1}_A \le S_m {\mathbf 1}_K$.  
Since for every $N \in \N$, 
\begin{align*}
S_N {\mathbf 1}_A & \le S_m(S_N{\mathbf 1}_K) \\
& \le 2m\sum_{j=-N}^{N-1}{\mathbf 1}_K \circ S^j+{\mathbf 1}_K \circ S^{-N-1}+
\{{\mathbf 1}_K \circ S^{-N-1}+{\mathbf 1}_K \circ S^{-N-2}\} \\ 
& \qquad + \dots + \{{\mathbf 1}_K \circ S^{-N-1}+\dots + {\mathbf 1}_K 
\circ S^{-N-m}\} +  {\mathbf 1}_K \circ S^N \\
& \qquad +\{{\mathbf 1}_K \circ S^N+{\mathbf 1}_K \circ S^{N+1}\}+ \dots + 
\{{\mathbf 1}_K \circ S^{N+1}+\dots + {\mathbf 1}_K \circ S^{N+m-2}\} \\
& \le 2m S_N {\mathbf 1}_K + m^2,
\end{align*}
the sequence $\left\{\frac{S_N{\mathbf 1}_A}{S_N{\mathbf 1}_K}\right\}_N$ is bounded 
uniformly on $K$. Since ${\mathfrak F}$ is countable, we may find a sequence 
$\{i_p\}_{p \in \N} \subset \N$ such that 
$\left\{\frac{S_{N_{i_p}}{\mathbf 1}_A(x_{i_p})}
{S_{N_{i_p}}{\mathbf 1}_K(x_{i_p})}\right\}_p$ 
converges for any $A \in {\mathfrak F}$. Let $\mu(A)$ denote the limit. 
Observe $\mu(K)=1$. The set function $\mu:{\mathfrak F} \to \R_+$ has a 
unique extension, denoted by $\mu$ again, to $\B$. 
Since $|S_{N_{i_p}}{\mathbf 1}_{S^{-1}A}(x_{i_p}) - 
S_{N_{i_p}}{\mathbf 1}_A(x_{i_p})| \le 1$ for any pair $(A,p) \in {\mathfrak F} 
\times \N$, it follows that $\mu$ is $S$\nobreakdash-invariant. 

We then see the second statement. It is easy to show the implications: 
$\eqref{uniform} \Rightarrow \eqref{pointwise} \Rightarrow \eqref{measure}$. 
Assume $\eqref{measure}$. Suppose $\nu$ is an $S$\nobreakdash-invariant 
Radon measure. The assumption of the proposition guarantees $\mu(K) \ne 0$ and 
$\nu(K) \ne 0$. By the ergodic decomposition 
(see for example\ \cite[2.2.9]{Aa}), there exist a probability space 
$(\Omega,\lambda)$ and $\sigma$\nobreakdash-finite ergodic measures 
$\{\rho_\omega|\omega \in \Omega\}$ on $X$ such that for any $B \in \B$, 
\begin{enumerate}[(i)]
\item
a function $\Omega \to \R, \omega \mapsto \rho_\omega(B)$ is 
measurable;
\item
$\nu(B) = \int_\Omega \rho_\omega(B) d \lambda(\omega)$.
\end{enumerate}
Since it follows from the ratio ergodic theorem 
(see for example \cite[2.2.1]{Aa}) that for any $B \in \B$, 
\[
\nu(B)=\int_\Omega \frac{\rho_\omega(K)}{\mu(K)}\mu(B) d\lambda(\omega) = 
\frac{\mu(B)}{\mu(K)}\int_\Omega \rho_\omega(K) d\lambda(\omega)=
\frac{\nu(K)}{\mu(K)}\mu(B),
\]
we obtain \eqref{unique}.

Assume \eqref{unique}. Let us show the first half of \eqref{uniform}. To 
the contrary, assume the existence of a point $x_0 \in X$ such that 
$\sharp (\Orb_S(x_0) \cap K) < \infty$. As is seen above, the counting 
measure on $\Orb_S(x_0)$ is an $S$\nobreakdash-invariant Radon measure. 
Observe $\overline{\Orb_S(x_0)}=\Orb_S(x_0)$. 
In view of \eqref{unique}, $\sharp (\Orb_S(x) \cap K) = \infty$ for any 
$x \in X \setminus \Orb_S(x_0)$, which yields an $S$\nobreakdash-invariant 
Radon measure singular to the counting measure. Hence, $S$ has singular 
invariant Radon measures, which contradicts \eqref{unique}. 
To show the second half of \eqref{uniform}, we shall see the implication: 
\eqref{uniform} $\Rightarrow c = \frac{\mu(A)}{\mu(K)}$, as follows. 
Assume \eqref{uniform}. Let $\mu$ denote the unique invariant Radon measure. 
There exist a probability 
space $(\Omega,\lambda)$ and $\sigma$\nobreakdash-finite ergodic measures 
$\{\rho_\omega|\omega \in \Omega\}$ such that 
$\mu(B)=\int_\Omega \rho_\omega(B) d \lambda(\omega)$. 
Since $\rho_\omega$ is a Radon measure, it follows that 
$\rho_\omega = c_\omega \mu$ with a constant $c_\omega$. It follows therefore 
that $\mu$ is ergodic, so that $c$ in \eqref{uniform} must equal 
$\frac{\mu(A)}{\mu(K)}$. Then, assume \eqref{uniform} is not the case under 
\eqref{unique}. There exist $A \in {\mathfrak F}$, $\epsilon > 0$, 
$\{x_m\}_{m \in \N} \subset K$ and $\{N_m\}_{m \in \N} \subset \N$ such that 
for each $m \in \N$, $S_{N_m}{\mathbf 1}_K(x_m) \ge m$ and 
$\left| \frac{S_{N_m} {\mathbf 1}_A(x_m)}{S_{N_m} {\mathbf 1}_K(x_m)}- 
\frac{\mu(A)}{\mu(K)}\right| \ge \epsilon$. 
As in the second paragraph, we may find an $S$\nobreakdash-invariant Radon 
measure $\nu$ such that 
$\left|\frac{\nu(A)}{\nu(K)} - \frac{\mu(A)}{\mu(K)}\right| \ge \epsilon$. 
This contradicts~\eqref{unique}.

Let us see the last statement. If $S$ is minimal, then 
$\bigcup_{i \in \Z}S^iA=X$ for any nonempty open set $A \subset X$, 
and hence, $\mu(A)>0$. Conversely, suppose any nonempty open set 
has a positive measure. If the orbit\nobreakdash-closure of some point is a 
proper subset of $X$, then we may find an invariant Radon measure which is 
singular with respect to $\mu$.  This is a contradiction. 
\end{proof}

\section{Uniform partitions}\label{constructions}

In this section, we first introduce the concept of uniformity for measurable 
sets and finite partitions, respectively. We then show that there exist so 
many uniform partitions as they generate the $\sigma$\nobreakdash-algebra $\C$.

\begin{definition}\label{uniform_infi}
A set $C \in \C_0$ is said to be {\em uniform} relative to a set $K \in \C_0$ 
if for any $\epsilon > 0$, there exists $m \in \N$ such that for a.e.\ $y \in K$, 
\begin{equation*}\label{sup_converg}
T_N{\mathbf 1}_K(y) \ge m \Rightarrow \left|\frac{T_N{\mathbf 1}_C(y)}
{T_N{\mathbf 1}_K(y)} - \frac{\nu(C)}{\nu(K)}\right|<\epsilon.
\end{equation*}

A finite partition $\alpha$ of $Y$ is said to be {\em uniform} relative to 
$K$ if all the sets in $\bigcup_{n \in \N} \alpha_{-n}^{n-1}$ of finite measure 
are uniform relative to $K$. 
\end{definition}

\begin{lemma}\label{symbolic_unique}
Let $\alpha$ be a finite partition of $Y$. Then, the following are equivalent:
\begin{enumerate}
\item
$(X_\alpha,S_\alpha)$ is a locally compact Cantor minimal system admitting a 
unique, up to scaling, invariant Radon measure;
\item
$\alpha$ is uniform relative to $K_\alpha$, and 
$\sharp(\Orb_{S_\alpha}(x) \cap K_{\phi_\alpha(\alpha)})=\infty$ for all 
$x \in X_\alpha$.
\end{enumerate}
\end{lemma}

\begin{proof}
This follows from Proposition~\ref{existence}. 
\end{proof}

Let $\alpha = \{A_1,A_2,\dots,A_N\}$ and $\beta=\{B_1,B_2,\dots,B_N\}$ be 
finite partitions of $Y$. We say that {\em the $\alpha$ distribution is within 
$\delta > 0$ of the $\beta$ distribution} if for every integer $i$ with 
$1 < i \le N$, 
\[
\left|\frac{\nu(A_i)}{\nu(K_{\alpha})} - \frac{\nu(B_i)}{\nu(K_{\beta})}\right| 
< \delta. 
\]
Let $w$ denote the $\alpha$\nobreakdash-name of a section 
$\{T^jy|0 \le j < N\}$. We say that on the section, {\em the 
$\alpha$ $(2n-1)$\nobreakdash-block distribution is within $\delta > 0$ of the 
$(\beta)_{-n+1}^{n-1}$ distribution} if for any word 
$v \in \L(\alpha) \setminus \{1^{2n-1}\}$ of length $2n-1$,
\[
\left|\frac{\sharp \{1 \le j \le N|w_{(j-n,j+n)} = v \}}
{\sharp\{1 \le j \le N|w_j \ne 1 \}} - 
\frac{\nu\left(\bigcap_{i=1}^{2n-1}T^{-(i-1)} 
B_{v_i}\right)}{\nu(K_{\beta})}\right| < \delta. 
\]

\begin{lemma}\label{initial_partition}
Let $\alpha_0$ be a finite partition of $Y$. Let $0 < \epsilon < 1$. 
Then, there exists a finite partition $\alpha$ of $Y$ such that 
\begin{enumerate}
\item
$d(\alpha_0,\alpha)<\epsilon$;
\item
$\alpha$ is uniform relative to $K_{\alpha}$;
\item\label{infinite_often}
$\sharp(\Orb_{S_\alpha}(x) \cap K_{\phi_\alpha(\alpha)})=\infty$ for all 
$x \in X_\alpha$.
\end{enumerate}
\end{lemma}

\begin{proof}

By constructing inductive steps, we will show that there exist finite partitions 
$\{\alpha_n\}_{n \in \N}$ of $Y$, $K_{\alpha_n}$\nobreakdash-standard towers 
${\mathfrak t}_k(n)$, $n \in \N$, $1 \le k \le n$, and 
integers $\hat{N}_n \ge 1$ which satisfy the following properties. 
\begin{enumerate}[(i)]
\item
$K_{\alpha_n} \subset K_{\alpha_{n-1}}$ for every $n \in \N$.
\item
For every pair $(n,k) \in \N \times \{1,2,\dots, n\}$, 
\begin{enumerate}
\item
$d((\alpha_n)_{-k+1}^{k-1},(\alpha_{n-1})_{-k+1}^{k-1}) < \frac{\epsilon}{2^n}$;
\item
${\mathfrak t}_k(n)$ refines $\alpha_n$ as partitions;
\item
${\mathfrak t}_{k+1}(n)$ refines ${\mathfrak t}_k(n)$ if $1 \le k < n$; 
\item
$B_k(n) \subset B_k(n+1)$ and $\nu(B_k(n+1) \setminus B_k(n)) < 
\frac{\epsilon}{2^{n+1}}$, where $B_k(n)=B({\mathfrak t}_k(n))$;
\item\label{good_distri}
on any fiber of each principal column of ${\mathfrak t}_n(n)$, 
the $\alpha_n$ $(2n-1)$\nobreakdash-block distribution is within 
$\frac{\epsilon}{2^n c_n}$ of the $(\alpha_n)_{-n+1}^{n-1}$ distribution, where $c_1=1$ and 
$c_n=(\sharp \alpha_0)^{2n-1}$ for $n \ge 2$.
\end{enumerate}
\item
There exist $m_k,M_k \in \N$ such that if $(n,k) \in \N \times \{1,2,\dots,n\}$ 
and if ${\mathfrak c}$ is a principal column of ${\mathfrak t}_k(n)$, then 
$m_k \le \sharp \{A \subset K_{\alpha_n}|A \in {\mathfrak c}\} \le M_k$ 
and $m_{k+1} > M_k$.

\end{enumerate}

{\bf Step~1:} 
Take $0 < \delta_1 < \frac{\epsilon}{2^2}$ so 
that if $d(\alpha_0,\beta) < \delta_1$ then the $\beta$ distribution is 
within $\frac{\epsilon}{2^2}$ of the $\alpha_0$ distribution. 
Applying Proposition~\ref{tower_lemma2}, one may find $N_1 \in \N$ 
such that if the height of each principal column of a 
$K_{\alpha_0}$\nobreakdash-standard tower is at least $N_1$ then those fibers 
of the tower, on which the $\alpha_0$ 1\nobreakdash-block distribution is within 
$\delta_1$ of the $\alpha_0$ distribution, cover at least 
$\nu(K_{\alpha_0})-\delta_1$ of $K_{\alpha_0}$. 

Using Lemma~\ref{tower_lemma1}, we may build a 
$K_{\alpha_0}$\nobreakdash-standard tower ${\mathfrak t}^\prime_1(1)$ all of 
whose principal columns have heights $N_1$ or $N_1 + 1$. By refining 
${\mathfrak t}^\prime_1(1)$ according to $\alpha_0$ if necessary, we may assume 
all the fibers on a column of ${\mathfrak t}^\prime_1(1)$ have the same 
$\alpha_0$\nobreakdash-name. Replace the $\alpha_0$\nobreakdash-name of any 
level in each {\em bad column} with $1$. The bad column means a 
principal column on which 
$\alpha_0$ $1$\nobreakdash-block distribution is {\em not} within $\delta_1$ of 
the $\alpha_0$ distribution. Unite all the levels of the bad columns with an 
infinite level of ${\mathfrak t}^\prime_1(1)$ into a new level. 
The change of $\alpha_0$\nobreakdash-names gives us a new finite 
partition $\alpha_1$ and yields 
$d(\alpha_0,\alpha_1) < \delta_1 < \frac{\epsilon}{2^2}$. 
The manipulation of uniting levels gives us a new 
$K_{\alpha_1}$\nobreakdash-standard tower ${\mathfrak t}_1(1)$. 
On the fiber of each principal column of ${\mathfrak t}_1(1)$, the $\alpha_1$ 
$1$\nobreakdash-block distribution is within $\delta_1$ of the $\alpha_0$ 
distribution. This together with $d(\alpha_0,\alpha_1) < \delta_1$ shows 
that \eqref{good_distri} holds with $n=1$. This allows us to find 
$\hat{N_1} \in \N$ such that if a section has at least $\hat{N}_1$ points in 
$K_{\alpha_1}$ then $\alpha_1$ $1$\nobreakdash-block distribution on the 
section is within $\frac{\epsilon}{2}$ of the $\alpha_1$ distribution. Since 
the following construction will guarantee that the 
$\alpha_n$\nobreakdash-name of any fiber of every tower 
${\mathfrak t}_1(n)$ with $n \ge 2$ will coincide with the 
$\alpha_1$\nobreakdash-name of a fiber of ${\mathfrak t}_1(1)$, any section 
having at least $\hat{N}_1$ points in $K_{\alpha_n}$ will have the 
$\alpha_n$ $1$\nobreakdash-block distribution within 
$\frac{\epsilon}{2}$ of the $\alpha_1$ distribution.

{\bf Step~2:} Fix a real number 
$\delta_2$ with $0 < \delta_2(N_1+1) < \frac{\epsilon}{2^3 c_2}$ so that 
if $d((\alpha_1)_{-1}^1,\beta) < 5 \delta_2$ then the $\beta$ distribution 
is within $\frac{\epsilon}{2^3 c_2}$ of the $(\alpha_1)_{-1}^1$ distribution 
and so that if $d(\alpha_1,\beta) < \delta_2$ then the $\beta$ 
distribution is within $\frac{\epsilon}{2^2}$ of the $\alpha_1$ distribution. 
Put 
\[
M_1=\max\{ \sharp \{A \in {\mathfrak c}|A \subset K_{\alpha_1}\}| 
{\mathfrak c} \textrm{ is a principal column of } {\mathfrak t}_1(1)\}. 
\]
Fix $m_2 > M_1$ with $\frac{\nu(K_{\alpha_1})}{m_2} < \delta_2$. 
Applying Proposition~\ref{tower_lemma2}, one may find $N_2 \in \N$ such that if 
the height of each principal column of a 
$K_{\alpha_1}$\nobreakdash-standard tower is at least $N_2$ then those fibers 
of the tower, on which the $\alpha_1$ $3$\nobreakdash-block distribution is 
within $\delta_2$ of $(\alpha_1)_{-1}^1$ distribution, cover at least 
$\nu(K_{\alpha_1}) - \delta_2$ of $K_{\alpha_1}$ and, in addition, each of 
the good fibers includes at least $m_2$ points in $K_{\alpha_1}$. 
In fact, if an atom $A \in (\alpha_1)_{-1}^1$ of finite measure is not 
included in $K_{\alpha_1}$, then, instead of $A$ itself, we have to apply the 
proposition to one of $TA$ and $T^{-1}A$ included in $K_{\alpha_1}$.

In virtue of Lemma~\ref{desired_refine}, we may build a 
$K_{\alpha_1}$\nobreakdash-standard tower ${\mathfrak t}^\prime_2(2)$ with base 
$B_2^\prime(2) \subset B_1(1)$ such that the height of each principal column of 
${\mathfrak t}_2^\prime(2)$ is between $N_2$ and $N_2+4(N_1+1)$. We may assume 
that all the fibers on a column of ${\mathfrak t}^\prime_2(2)$ have the same 
$\alpha_1$\nobreakdash-name. Replace with $1$ the $\alpha_1$\nobreakdash-name 
of any level of any bad column of ${\mathfrak t}^\prime_2(2)$. This change of 
names gives us a new finite partition $\alpha_2$ and guarantees 
$d(\alpha_1,\alpha_2)<\delta_2$. Then, unite all the levels in the 
bad columns with an infinite level of ${\mathfrak t}^\prime_2(2)$ into a new 
level. This change 
of ${\mathfrak t}_2^\prime(2)$ results a new $K_{\alpha_2}$\nobreakdash-standard 
tower ${\mathfrak t}_2(2)$. Put $P_2(2)=B_2(2) \cap K_{{\mathfrak t}_2(2)}$.
We then have $\nu(P_2(2))m_2 \le \nu(K_{\alpha_1})$. 
Counting the case where $T$ or $T^2$ makes subsets of an atom of 
$(\alpha_2)_{-1}^1$ of finite measure go through the top to the base of 
${\mathfrak t}_2(2)$, we may see that 
\begin{equation}\label{2-blocks}
d((\alpha_1)_{-1}^1,(\alpha_2)_{-1}^1)<\delta_2 + 4\nu(P_2(2))<\delta_2 + 
\frac{4\nu(K_{\alpha_1})}{m_2} < 5 \delta_2. 
\end{equation}
Decomposing the principal columns of ${\mathfrak t}_2(2)$ into subcolumns of 
columns of ${\mathfrak t}_1(1)$ and putting down the subcolumns, we obtain a 
new $K_{\alpha_2}$\nobreakdash-standard tower ${\mathfrak t}_1(2)$. In fact, 
this construction of ${\mathfrak t}_1(2)$ has to be executed so that different 
columns of ${\mathfrak t}_1(2)$ are subcolumns of different columns of 
${\mathfrak t}_1(1)$. This additional manipulation is done by uniting 
subcolumns whose bases are included in the base of a column of 
${\mathfrak t}_1(1)$. 
Since $B_1(2)$ is the union of $B_1(1)$ and the bad columns of 
${\mathfrak t}_2^\prime(2)$, we have 
$\nu(B_1(2) \setminus B_1(1)) < \delta_2(N_1+1) < \frac{\epsilon}{2^2}$. 
On the fiber of each principal column of  ${\mathfrak t}_2(2)$, 
the $\alpha_2$ $3$\nobreakdash-block distribution is within $\delta_2$ of the 
$(\alpha_1)_{-1}^1$ distribution. This together with \eqref{2-blocks} 
implies that \eqref{good_distri} holds with $n=2$. Thus, there 
exists $\hat{N_2} \in \N$ such that any section having 
at least $\hat{N}_2$ points in $K_{\alpha_2}$ has the $\alpha_2$ 
$3$\nobreakdash-block distribution within $\frac{\epsilon}{2^2c_2}$ of the 
$(\alpha_2)_{-1}^1$ distribution, and hence, has the $\alpha_2$ 
$1$\nobreakdash-block distribution within $\frac{\epsilon}{2^2}$ of the 
$\alpha_2$ distribution. Since the $\alpha_n$\nobreakdash-name of any 
fiber of any tower ${\mathfrak t}_n(2)$ with $n \ge 3$ will coincide with 
the $\alpha_2$\nobreakdash-name of a fiber of ${\mathfrak t}_2(2)$, 
any section having at least $\hat{N}_2$ points in $K_{\alpha_n}$ will 
have the $\alpha_n$ $1$\nobreakdash-block (resp.\ $3$\nobreakdash-block) 
distribution within $\frac{\epsilon}{2^2}$ of the $\alpha_2$ 
(resp.\ $(\alpha_2)_{-1}^1$) distribution.

In the following steps, repeating arguments of Step~2 with suitably 
arranged parameters, we may obtain the desired sequences 
$\{\alpha_n\}_n$, ${\mathfrak t}_k(n)$ and 
$\{\hat{N}_n\}_n$. There exist a finite partition $\alpha$ of $Y$ such that 
$\lim_{n \to \infty}d(\alpha,\alpha_n)=0$. 
Then, $d(\alpha_0,\alpha) \le \sum_{i=0}^\infty d(\alpha_i,\alpha_{i+1}) \le 
\frac{\epsilon}{2} < \epsilon$. 
For each $k \in \N$, there exists $B_k \in \C$ such that $B_k(n) \uparrow B_k$ 
as $n \to \infty$. Since for each pair $(n,k) \in \N \times \{1,2,\dots,n\}$, 
${\mathfrak t}_k(n+1)$ is obtained from ${\mathfrak t}_k(n)$ by uniting 
subcolumns of ${\mathfrak t}_k(n)$ with the infinite level of 
${\mathfrak t}_k(n)$ into a new level, it follows that for any pair 
$(n_0,k) \in \N \times \{1,2,\dots,n_0\}$ and for any level $F_{n_0} \in 
{\mathfrak t}_k(n_0)$ of finite measure, there exists a unique sequence 
$\{F_n \in {\mathfrak t}_k(n)|n \ge n_0\}$ of decreasing levels. If 
$\bigcap_{n \ge n_0}F_n \ne \emptyset$, then the intersection works as a level 
of a tower ${\mathfrak t}_k$ with base $B_k$. The infinite level of 
${\mathfrak t}_k$ is the union over $n$ of the infinite levels of 
${\mathfrak t}_k(n)$. Since for any pair $(n,k) \in \N \times \{1,2,\dots,n\}$, 
the $\alpha$\nobreakdash-name of any fiber of ${\mathfrak t}_k$ 
coincides with the $\alpha_n$\nobreakdash-name of a fiber of 
${\mathfrak t}_k(n)$, we can show the uniformity of $\alpha$ as follows. Let 
$\delta > 0$ and $k \in \N$. Choose an 
integer $n \ge k$ so that $\frac{\epsilon}{2^{n-1}} < \delta$. Then, in virtue of 
\eqref{good_distri}, any section having at least $\hat{N}_n$ points in $K_\alpha$ has 
the $(2k-1)$\nobreakdash-distribution within $\frac{\epsilon}{2^n}$ of the 
$(\alpha_n)_{-k+1}^{k-1}$ distribution. This together with 
$d((\alpha_n)_{-k+1}^{k-1},(\alpha)_{-k+1}^{k-1})< \frac{\epsilon}{2^n}$ 
leads to the fact that 
any set in $(\alpha)_{-k+1}^{k-1}$ of finite measure is uniform relative to 
$K_\alpha$; it might be necessary to choose again $n$ so larger that the 
$(\alpha_n)_{-k+1}^{k-1}$ distribution is sufficiently close to the 
$(\alpha)_{-k+1}^{k-1}$ distribution. 

Since $K_\alpha \subset K_{\alpha_n} \subset K_{{\mathfrak t}_k(n)}$ for any 
pair $(n,k) \in \N \times \{1,2,\dots,n\}$, we have $K_\alpha \subset 
\bigcap_{n=1}^\infty K_{{\mathfrak t}_k(n)}=K_{{\mathfrak t}_k}$ for any $k \in 
\N$. In view of 
the definition of ${\mathfrak t}_k$, it is not hard to see that 
${\mathfrak t}_{k+1}$ refines ${\mathfrak t}_k$ for each $k \in \N$, and that 
the number of those levels in a principal column of ${\mathfrak t}_k$ which are 
included in $K_\alpha$ is between $m_k$ and $M_k$. In particular, 
$\{{\mathfrak t}_k\}_k$ is a refining sequence of 
$K_\alpha$\nobreakdash-standard towers. Then, Property~\eqref{infinite_often} 
follows from Lemma~\ref{often_lemma}. This completes the proof. 
\end{proof}

Given finite partitions $\alpha$ and $\beta$ of $Y$, we write 
$\alpha \succcurlyeq \beta$ if $\alpha$ is finer than $\beta$ and if $K_\alpha = 
K_\beta$. 

\begin{lemma}\label{refining_uniform}
Suppose a finite partition $\beta$ is uniform relative to $K_\beta$. 
Suppose a finite partition $\alpha_0$ is such that 
$\alpha_0 \succcurlyeq \beta$. Let $0 < \epsilon < 1$. Then, there exists a 
uniform partition $\alpha \succcurlyeq \beta$ relative to $K_\alpha$ which 
satisfies $d(\alpha_0,\alpha) < \epsilon$. 
\end{lemma}

\begin{proof}
By constructing inductive steps in a similar way to the proof of 
Lemma~\ref{initial_partition}, we shall obtain a sequence 
$\{\alpha_n\}_{n \in \N}$ of finite partitions whose limit $\alpha$ satisfies 
the desired properties. However, we have to adjust some aspects of the proof. 
Without putting the name $1$ on any level in any bad column, without 
putting together the bad columns with an infinite level, we will 
copy the name of a good fiber on bad fibers. Towers $\{{\mathfrak t}_k(n)|
1 \le k \le n, n \in \N \}$ are built 
so as to be measurable with respect to the $\sigma$\nobreakdash-algebra $\B$ 
generated by $\bigcup_{n \in \N} (\beta)_{-n}^n$, and so that 
${\mathfrak t}_k(n)$ does not depend on $n$. Also, for any $n \in \N$, any 
change of $\alpha_n$\nobreakdash-names are made among names which 
atoms associated with are included in a unique atom of $\beta$.

{\bf Step~1.} Put $K=K_\beta$. 
Take $0 < \delta_1 < \frac{1}{9}\epsilon \min\{1,\nu(K)\}$ so that 
if $d(\alpha_0,\beta) < \delta_1$ then the $\beta$ distribution is within 
$\frac{\epsilon}{9}$ of the $\alpha_0$ distribution. There exists $N_1 \in \N$ such 
that if the height of each principal column of a $K$\nobreakdash-standard tower 
is at least $N_1$ then those fibers of the tower on which $\alpha_0$ 
$1$\nobreakdash-block distributions are within $\delta_1$ of the $\alpha_0$ 
distribution cover at least $\nu(K)-\delta_1$ of $K$. 
Since the set $C$ in the proof of Lemma~\ref{tower_lemma1} 
may be chosen from $\B$, there exists a $\B$\nobreakdash-measurable, 
$K$\nobreakdash-standard tower ${\mathfrak t}_1(1)$ such that the height of 
each principal column is $N_1$ or $N_1+1$. Refine ${\mathfrak t}_1(1)$ according 
to $\beta$, so that all fibers on a fixed column have a $\beta$\nobreakdash-name 
in common.

If a fiber of a principal column ${\mathfrak c}$ of ${\mathfrak t}_1(1)$ 
has a good $\alpha_0$\nobreakdash-name, then copy the good 
$\alpha_0$\nobreakdash-name on any other fibers of ${\mathfrak c}$. 
This change of name makes $\alpha_0$ change at most $\delta_1$ in 
$d$ because of the assumption $K_{\alpha_0}=K_\beta$. 
If no fibers of the principal column ${\mathfrak c}$ have good 
$\alpha_0$\nobreakdash-names, then we do not change the 
$\alpha_0$\nobreakdash-name of any fiber of ${\mathfrak c}$. Instead, we 
will make use of the uniformity of the union $R({\mathfrak c})$ of those 
levels in ${\mathfrak c}$ which are included in $K$. Let $R$ denote the 
union of all such unions $R({\mathfrak c})$. 

The change of $\alpha_0$\nobreakdash-names, described in the preceding paragraph, 
yields a new finite partition $\alpha_1$. It follows from the way of changing 
names that $d(\alpha_0,\alpha_1) < \delta_1$ and 
$\alpha_1 \succcurlyeq \beta$. If a fiber of a principal column of 
${\mathfrak t}_1(1)$ is disjoint from $R$, then on the fiber the 
$\alpha_1$ $1$\nobreakdash-block distribution is within 
$\frac{\epsilon}{9}$ of the $\alpha_0$ distribution. This allows 
us to find $M_1 \in \N$ such that for any set 
$A \in \alpha_1$ of finite measure and for a.e.\ $y \in K$,
\begin{equation*}\label{0to1}
T_m{\mathbf 1}_{K \setminus R}(y) \ge M_1 \Rightarrow 
\left| \frac{T_m{\mathbf 1}_{A \setminus R}(y)}
{T_m{\mathbf 1}_{K \setminus R}(y)} - 
\frac{\nu(A^\prime)}{\nu(K)}\right| < \frac{\epsilon}{9},
\end{equation*}
where $A^\prime$ is an atom of $\alpha_0$ having the same index as $A$. 
Since $R$ is uniform relative to $K$, there exists $\hat{N}_1 \in \N$ with 
$(1-\frac{\epsilon}{9})\hat{N}_1 \ge M_1$ such that for a.e.\ $y \in K$, 
\begin{equation*}\label{less_freq_1}
T_m{\mathbf 1}_K(y) \ge \hat{N}_1 \Rightarrow 
\frac{T_m{\mathbf 1}_R(y)}{T_m{\mathbf 1}_K(y)} < \frac{\epsilon}{9}.
\end{equation*}
These properties together with the fact that 
$d(\alpha_0,\alpha_1) < \delta_1$ implies that for any set $A \in \alpha_1$ of 
finite measure and for a.e.\ $y \in K$, if 
$T_m{\mathbf 1}_K(y) \ge \hat{N}_1$ then 
\begin{align*}
\left| \frac{T_m{\mathbf 1}_A(y)}
{T_m{\mathbf 1}_K(y)} - \frac{\nu(A)}{\nu(K)}\right| & \le 
\left| \frac{T_m{\mathbf 1}_{A \setminus R}(y)}{T_m{\mathbf 1}_K(y)} - 
\frac{\nu(A^\prime)}{\nu(K)}\right| + \frac{T_m{\mathbf 1}_R(y)}
{T_m{\mathbf 1}_K(y)} + \left| \frac{\nu(A^\prime)}{\nu(K)} - 
\frac{\nu(A)}{\nu(K)}\right| \\
& < \left| \frac{T_m{\mathbf 1}_{A \setminus R}(y)}{T_m{\mathbf 1}_{K \setminus 
R}(y)} \left( 1 - \frac{T_m{\mathbf 1}_R(y)}{T_m{\mathbf 1}_K(y)} \right) - 
\frac{\nu(A^\prime)}{\nu(K)} \right|  + \frac{\epsilon}{9} + \delta_1 \\
& \le \left| \frac{T_m{\mathbf 1}_{A \setminus R}(y)}{T_m{\mathbf 1}_{K 
\setminus R}(y)} - \frac{\nu(A^\prime)}{\nu(K)}\right| + \frac{\epsilon}{9} 
+ \frac{2}{9}\epsilon < \frac{\epsilon}{2}.
\end{align*}
The author would like to leave the remainder of the inductive steps to the reader. 
\end{proof}

\begin{theorem}\label{J-K}
Let $(Y,\C,\nu,T)$ be an ergodic, infinite measure\nobreakdash-preserving system. 
Then, there exists a locally compact Cantor minimal system $(X,S)$ admitting a 
unique, up to scaling, invariant Radon measure $\mu$ such that $(Y,\C,\nu,T)$ is 
isomorphic to $(X,\B_\mu,\mu,S)$. 
\end{theorem}

\begin{proof}

Lemma~\ref{initial_partition} allows us to find a finite partition $\alpha_1$ 
of $Y$, which is uniform relative to $K_{\alpha_1}$, such that 
$\sharp(\Orb_{S_{\alpha_1}}(x) \cap K_{\phi_{\alpha_1}(\alpha_1)}) = \infty$ 
for all $x \in X_{\alpha_1}$. Using Lemma~\ref{refining_uniform} inductively, 
find a sequence $\{\alpha_n\}_{n \ge 2}$ of finite partitions such that 
\begin{enumerate}
\item
$\alpha_m \succcurlyeq \alpha_n$ for all integers $m,n$ with $m \ge n \ge 1$;
\item
$\alpha_n$ is uniform relative to $K_{\alpha_1}$ for all integers $n \ge 2$;
\item\label{generates}
$\bigcup_{n=1}\bigvee_{k=1}^\infty (\alpha_n)_{-k}^k$ generates the 
$\sigma$\nobreakdash-algebra $\C$.
\end{enumerate}
Let $X_n$ be the complement of a unique fixed point in 
$\hat{X}_n:=\hat{X}_{\alpha_n}$. Let $S_n$ and $\mu_n$ denote the 
restrictions to $X_n$ of $\hat{S}_n:=\hat{S}_{\alpha_n}$ and 
$\hat{\mu}_n:=\hat{\mu}_{\alpha_n}$, respectively. It follows from 
Lemma~\ref{symbolic_unique} that $\mu_n$ is a unique, up to scaling, 
invariant Radon measure of a locally compact Cantor minimal system 
$(X_n,S_n)$. Let $(\hat{X},\hat{S})$ denote the inverse limit of an 
inverse system $(\hat{X}_n,\hat{S}_n,\phi_{n,m})$, where 
$\phi_{n,m}=\phi_{\alpha_n,\alpha_m}$, and $(X,S)$ the restriction of 
$(\hat{X},\hat{S})$ to the complement of a unique fixed point. 
Let $\hat{\mu}$ denote an $\hat{S}$\nobreakdash-invariant measure constructed 
from $\{\hat{\mu}_n\}_n$ as in Section~\ref{sec_top}, and 
$\mu$ the restriction of $\hat{\mu}$ to $X$. In view of Lemma~\ref{unique_Radon}, 
$(X,S)$ is a locally compact Cantor minimal system admitting a unique, up to 
scaling, invariant Radon measure $\mu$. 
Define a map $\phi:(Y,\nu,\C) \to (\hat{X},\hat{\B}_{\hat{\mu}},\hat{\mu})$ by 
$y \mapsto (\phi_n(y))_n$, where $\phi_n=\phi_{\alpha_n}$. 
Let $p_n:\hat{X} \to \hat{X}_n$ be the projection. Since 
$p_n \circ \phi = \phi_n$ for every $n \in \N$, $\phi$ is measurable. 
We also have $(\nu \circ \phi^{-1}) \circ p_n^{-1} = \nu \circ 
{\phi_n}^{-1} = \hat{\mu}_n$ for every $n \in \N$, so that 
$\nu \circ \phi^{-1} = \hat{\mu}$. In view of \eqref{generates}, $\phi$ is 
injective. Since a measurable subset $\{(\phi_n(y))_n|y \in Y\}$ of 
$\hat{X}$ has full measure, $\phi$ is surjective. It is readily verified that 
$\phi \circ T = \hat{S} \circ \phi$. This completes the proof. 
\end{proof}

\bibliographystyle{amsplain}
\bibliography{Mybib}

\end{document}